\documentclass[12pt,a4paper]{amsart}
\usepackage[top=35mm, bottom=35mm, left=30mm, right=30mm]{geometry}
\usepackage[colorlinks=true,citecolor=blue]{hyperref}
\usepackage{mathptmx}
\usepackage{eucal}
\usepackage{graphicx}
\usepackage{mathrsfs}
\usepackage{amssymb}
\usepackage{amsmath}
\usepackage{amsthm}
\usepackage{amscd}
\usepackage{xcolor}
\usepackage{verbatim}

\newtheorem{thm}{Theorem}[section]
\newtheorem{cor}[thm]{Corollary}

\newtheorem{prop}[thm]{Proposition}

\theoremstyle{definition}
\newtheorem{defn}[thm]{Definition}
\newtheorem{rem}[thm]{Remark}

\numberwithin{equation}{section}

\newcommand{\N}{\mathbb{N}}

\newcommand{\ep}{\epsilon}

\newcommand{\ra}{\rightarrow}

\def \B {\mathcal B}
\numberwithin{equation}{section}
\newcommand{\dd}{\mathop{}\!\mathrm{d}}

\input xy
\xyoption{all}

\begin{document}

\title{Measure-theoretic mean equicontinuity and bounded complexity}
\author{Tao Yu}
\address[T. Yu]{Shanghai center for mathematical sciences, Fudan University, Shanghai, 200433, P.R. China}
\email{ytnuo@mail.ustc.edu.cn}

\begin{abstract}
Let $(X,\B,\mu,T)$ be a measure preserving system.
We say that a function $f\in L^2(X,\mu)$ is $\mu$-mean equicontinuous
if for any $\ep>0$ there is $k\in \N$ and measurable sets ${A_1,A_2,\cdots,A_k}$
with $\mu\left(\bigcup\limits_{i=1}^k A_i\right)>1-\ep$ such that whenever $x,y\in A_i$ for some $1\leq i\leq k$, one has
\[ \limsup_{n\to\infty}\frac{1}{n}\sum_{j=0}^{n-1}|f(T^jx)-f(T^jy)|<\ep.
\]
Measure complexity with respect to $f$ is also introduced.
It is shown that $f$ is an almost periodic function
 if and only if $f$ is $\mu$-mean equicontinuous if and only if $\mu$ has bounded complexity with respect to $f$.

Ferenczi studied measure-theoretic complexity
using $\alpha$-names of a partition and the Hamming distance. He proved that
if a  measure preserving system  is ergodic, then the complexity function is bounded if and only if the system
has discrete spectrum.
We show that this result holds without the assumption of ergodicity.
\end{abstract}
\maketitle

\section{Introduction}
Throughout this paper,
 a topological dynamical system (t.d.s.\ for short)
is a pair $(X,T)$, where $X$ is a compact metric space with a metric $d$
and $T$ is a  homeomorphism  from $X$ to itself,
and a measure preserving system (m.p.s.\ for short)
is a quadruple $(X,\mathcal{B},\mu,T)$,
where $(X,\mathcal{B},\mu)$ is a Lebesgue space
and $T$ is an invertible measure-preserving transformation from $X$ to itself.
There is a natural connection between topological dynamical systems
and measure preserving systems.
If $(X,T)$ is a t.d.s., then a probability measure $\mu$ on $(X,\mathcal{B}_X)$
is $T$-invariant if and only if $(X,\mathcal{B}_X,\mu,T)$ is a m.p.s.,
where $\mathcal{B}_X$ is the Borel $\sigma$-algebra of $X$.

A t.d.s.\ $(X,T)$ is called equicontinuous
if for any $\varepsilon>0$ there exists $\delta>0$
such that $d(T^nx,T^ny)<\varepsilon$ for any $n\geq 0$ whenever $d(x,y)<\delta$.
The analogous concept of equicontinuity for m.p.s.\
is the discrete spectrum property.
In the study of discrete spectrum, Fomin introduced the concept
of  stability in the mean in the sense of Lyapunov
(mean-L-stability for short) in \cite{F51}.
It was shown in \cite{LTY} that mean-L-stability
is equivalent to mean equicontinuity.
In \cite{LTY}, Li, Tu and Ye showed that
every ergodic measure in a mean equicontinuous system
has discrete spectrum.

In \cite{HLY}, Huang, Lu and Ye studied equicontinuity
for invariant measures of t.d.s..  They showed that for an invariant measure $\mu$ on a t.d.s.\ $(X,T)$, if $(X,T,\mu)$ is $\mu$-equicontinuous
then $\mu$ has discrete spectrum.
Following this idea, in \cite{Felipe} Garc\'ia-Ramos introduced
mean equicontinuity
for an ergodic invariant measure $\mu$ of a t.d.s.\ $(X,T)$
and showed that $(X,T,\mu)$ is $\mu$-mean equicontinuous if and only if it has discrete spectrum,
see also \cite{L16} for another proof.
It should be noticed that recently the authors in \cite{HLTXY,HWY}
showed that for an invariant measure $\mu$ on a t.d.s.\ $(X,T)$,  $(X,T,\mu)$ is $\mu$-mean equicontinuous if and only if it has discrete spectrum.

Complexity function is very useful to describe equicontinuity.
In \cite{BHM}, Blanchard et al.\ studied topological complexity via the complexity function of an open cover and
showed that the complexity function is bounded for any open cover if and only if the system is equicontinuous.
In \cite{HLTXY,HWY}, Huang et al.\ studied topological and measure-theoretic complexity via a sequence of metrics induced by a metric and
showed that an invariant  measure $\mu$ on $(X,T)$ has bounded  complexity
with respect to  $\bar{d}_n$ if and only if $(X,T)$ is $\mu$-mean equicontinuous if and only if it has discrete spectrum,
where $\bar d_n(x,y)=\frac{1}{n}\sum_{i=0}^{n-1}d(T^ix,T^iy)$.

Ferenczi \cite{Fer} studied  measure-theoretic complexity of
a m.p.s $(X,\B,\mu,T)$ using $\alpha$-names of a partition and the Hamming distance. He proved that if a m.p.s  is ergodic
then the complexity function is bounded
if and only if the system has discrete spectrum. In this paper, we will show  that for an invariant measure $\mu$ on $(X,T)$,
$(X,T)$ is $\mu$-mean equicontinuous if and only if the complexity of
$(X,\B_X,\mu,T)$ using $\alpha$-names of a partition and the Hamming distance is bounded.
 Huang et al. \cite{HLTXY} proved that $(X,T)$ is $\mu$-mean equicontinuous if and only if it has discrete spectrum,
we know Ferenczi's result also holds for invariant measure.

Following the idea in \cite{F81},
Garc\'ia-Ramos and Marcus introduced measure-theoretic mean equicontinuity with respect to a function \cite{GM15}.
Let $(X,T)$ be a t.d.s. with metric $d$ and $\mu$ be a $T$-invariant measure on $X$.
We say that
a function $f\in L^1(X,\mu)$ is $\mu$-mean equicontinuous
if for every $\tau>0$
there exists a compact subset $K$ of $X$ with $\mu(K)>1-\tau$ such that  for any $\varepsilon>0$ there exists $\delta>0$
such that
\[\limsup_{n\to\infty} \frac{1}{n}\sum_{i=0}^{n-1}|f(T^ix)-f(T^iy)|<\varepsilon\]
whenever $d(x,y)<\delta$ with $x,y\in K$.
It was shown in \cite{GM15} that for an ergodic measure $\mu$ on $(X,T)$, $f\in L^2(X,\mu)$ is $\mu$-mean equicontinuous if and only if $f$ is an  almost periodic function.

In this paper, we use complexity function to describe measure-theoretic mean equicontinuity with respect to a function.
For a function $f\in L^1(X,\mu)$, we define a sequence of pseudo-metrics $\bar f_n(x,y)=\frac{1}{n}\sum_{i=0}^{n-1}|f(T^ix)-f(T^iy)|$.
We consider the measure complexity with respect to $\{\bar f_n\}$.
 It is shown that for an invariant measure $\mu$ on $(X,T)$ and $f\in L^2(X,\mu)$,
  $\mu$ has bounded complexity with respect to $\{\bar f_n\}$
if and only if $f$ is $\mu$-mean equicontinuous if and only if
 $f$ is an  almost periodic function.

We show that these results also hold for a m.p.s $(X,\B,\mu,T)$.
The only difficulty to overcome is that the definition of $\mu$-mean equicontinuity of a function
 relies on a given metric on $X$,
even though it does not depend on the metric.
We will give a purely measure-theoretical definition of $\mu$-mean equicontinuity of a function
which is equal to the ordinary one in t.d.s.,
which answers the question in \cite{GM15}.
Let $(X,\B,\mu,T)$ be a m.p.s..
We say that a function $f\in L^1(X,\mu)$ is $\mu$-mean equicontinuous in the measure-theoretic sense
if for any $\ep>0$ there is $k\in \N$ and  measurable sets ${A_1,A_2,\cdots,A_k}$
with $\mu\left(\bigcup\limits_{i=1}^k A_i\right)>1-\ep$ such that whenever $x,y\in A_i$ for some $1\leq i\leq k$, one has
\[ \limsup_{n\to\infty}\frac{1}{n}\sum_{i=0}^{n-1}|f(T^ix)-f(T^iy)|<\ep.
\]

The structure of the paper is the following.  In Section 2, we recall some basic notions which we will use in the paper.
In Section 3, we show that for a m.p.s. $(X,\B,\mu,T)$, complexity function using $\alpha$-names of a partition and the Hamming distance
is bounded if and only if the system has discrete spectrum,
which generalize Ferenczi's result without the assumption of ergodicity.
In Section 4, we give a purely measure-theoretic definition of the measurable mean equicontinuity of a function,
and show that it is equal to the ordinary one in t.d.s..
We will also prove for a t.d.s.\ $(X,T)$ with an invariant measure $\mu$,
a function $f\in L^1(X,\mu)$ is $\mu$-mean equicontinuous if and only if $\mu$ has bounded complexity with respect to $\{\bar f_n\}$.
In Section 5, we show that $f\in L^2(X,\mu)$ is $\mu$-mean equicontinuous if and only if $f$ is an almost periodic function.

\medskip
\noindent {\bf Acknowledgments.}
I would like to thank Professors Xiangdong Ye, Jian Li and Guohua Zhang for useful suggestions.

\section{Preliminaries}

In this section we recall some notions and aspects of dynamical systems.
Let $A$ be a finite subset of a topological space. Denote by $\#(A)$ the number of elements of $A$.

Let $F$ be a subset of $\mathbb{Z}_+$. Define the \emph{upper density $\overline{D}(F)$ of $F$} and \emph{lower density $\underline{D}(F)$ of $F$} by
\[\overline{D}(F)=\limsup_{N\to\infty} \frac {1}{N} \#(F\cap [0,N-1])\]
\[\underline{D}(F)=\liminf_{N\to\infty} \frac {1}{N} \#(F\cap [0,N-1]).\]

\subsection{$\mu$-mean equicontinuity}
Let $(X,T)$ be a t.d.s. with a metric $d$.
Denote the collection of all invariant measures on $(X,T)$ by $M(X,T)$, and all ergodic measures on $(X,T)$ by $M^e(X,T)$.
We say that a subset $K$ of $X$ is a \emph{mean equicontinuous set}
if for every $\ep>0$,
there exists a $\delta>0$ such that for
all $x,y\in K$ with $d(x,y)<\delta$ we have \[\underset{n\to\infty}{\limsup}\frac{1}{n}\sum_{i=0}^{n-1} d(T^ix,T^iy)<\ep.\]
Let $\mu\in M(X,T)$. We say that $(X,T)$ is \emph{$\mu$-mean equicontinuous} if for every $\tau>0$
there exists a mean equicontinuous set $K$ of $X$ with $\mu(K)>1-\tau$.
Let $f\in L^1(X,\mu)$.
We say that a subset $K$ of $X$ is $f$-\emph{mean equicontinuous set}
if for every $\ep>0$,
there exists a $\delta>0$ such that \[\underset{n\to\infty}{\limsup}\frac{1}{n}\sum_{i=0}^{n-1} |f(T^ix)-f(T^iy)|<\ep\] for
all $x,y\in K$ with $d(x,y)<\delta$.
We say that $f$ is \emph{$\mu$-mean equicontinuous} if for every $\tau>0$
there exists a $f$-mean equicontinuous set $K$ of $X$ with $\mu(K)>1-\tau$.

\subsection{Lebesgue Space}
A measurable space $(X,\B)$ is called a  $\emph{Borel space}$ if it is isomorphic to $(Y,\B_Y)$ where $Y$ is a subset of the Cantor set.
A topological space $X$ is called $\emph{Polish}$ if there exists a metric on $X$ such that the metric topology coincides
with the original topology and with respect to this metric $X$ is complete and second countable.
A Borel space $(X,\B)$ is called a  $\emph{standard Borel space}$ if it is isomorphic to $(Y,\B_Y)$ where $Y$ is a Polish space.
We call a measure space $(X,\B,\mu)$ a $\emph{Lebesgue space}$ if $(X,\B)$ is a standard Borel space and $\mu$ a regular probability measure on $\B$.

\begin{thm}\label{0625}
If $(X,\B,\mu,T)$ is a m.p.s. where $(X,\mathcal{B},\mu)$ is a Lebesgue space, then there is a t.d.s. $(Y,S)$ such that $\pi:(Y,\B_Y,\nu,S)\ra (X,\B,\mu,T)$ is measure-theoretically isomorphic, where $\nu$ is an $S$-invariant Borel probability measure and $\B_Y$ is the $\sigma$-algebra
consisting of Borel subsets of $Y$.
\end{thm}
\subsection{Discrete spectrum and almost periodic functions}

Let $(X, \B,\mu, T )$ be an invertible m.p.s.
An {\it eigenfunction } is some non-zero
function $f\in L^2(X,\mu)$ such that $Uf = f\circ T =\lambda f$ for some $\lambda \in \mathbb{C}$.
In this case, $\lambda$ is
called the {\it eigenvalue} corresponding to $f$. It is easy to see every eigenvalue has norm one,
that is $|\lambda| = 1$. If $f\in L^2(X,\mu)$ is an eigenfunction, then $\text{cl}\{U^nf : n\in \mathbb{Z}\}$ is  compact in $L^2(X,\mu)$.
 Generally, we say that $f$ is {\it almost periodic} if $\text{cl}\{U^nf : n\in \mathbb{Z}\}$ is compact
in $L^2(X,\mu)$. It is well known that the set of all bounded almost periodic functions forms a
$U$-invariant and conjugation-invariant subalgebra of $L^2(X,\mu)$ (denoted by $A_c$). The set of all
almost periodic functions is just the closure of $A_c$ (denoted by $H_c$), and is also spanned by
the set of eigenfunctions. $(X, \B,\mu, T )$ is said  to have {\it discrete spectrum} if $L^2(X,\mu)$
is spanned by the set of eigenfunctions, that is $H_c = L^2(X,\mu)$.

\section{Complexity function and discrete spectrum}

Ferenczi \cite{Fer} studied  measure-theoretic complexity of
a m.p.s $(X,\B,\mu,T)$ using $\alpha$-names of a partition and the Hamming distance. He proved that if a m.p.s  is ergodic
then the complexity function is bounded
if and only if the system has discrete spectrum.
In this section, we show that for an invariant measure $\mu$ on $(X,T)$,
$(X,T)$ is $\mu$-mean equicontinuous if and only if the complexity of
$(X,\B_X,\mu,T)$ using $\alpha$-names of a partition and the Hamming distance is bounded.
It should be noticed that Huang et al.  \cite[Theorem 4.2, 4.3, 4.6]{HLTXY} proved that $(X,T)$ is $\mu$-mean equicontinuous if and only if it has discrete spectrum.

Let $(X,\mathcal{B},\mu,T)$ be a m.p.s. and $\alpha=\{A_1,\dotsc,A_l\}$ be a measurable partition of $X$.
For a point $x\in X$, the $\alpha$-name $\alpha(x)$ is the bi-infinite
sequence $\alpha_i(x)$ where $\alpha_i(x)=h$ whenever $T^nx \in A_h$.
For $x,y\in X$ and $n\geq 1$, let
\[H^\alpha_n(x,y)=\frac{1}{n}\#\{0\leq i\leq n-1\colon \alpha_i(x)\neq \alpha_i(y)\} .\]
For $x\in X$, $n\geq 1$ and $\varepsilon>0$, let
\[ B_n^\alpha(x,\varepsilon)=\{y\in X\colon H^\alpha_n(x,y)<\varepsilon\} \]
and
\[
K(n,\alpha,\varepsilon )=\min\{ \#(F):  \mu(\bigcup_{x\in F }B_n^\alpha(x,\varepsilon) )>1-\varepsilon\}.
\]
It was proved by Ferenczi \cite[Proposition 3]{Fer} that
if $\mu$ is ergodic, $(X,\mathcal{B},\mu,T)$ has discrete spectrum
if and only if for any finite partition $\alpha$ of $X$ and $\varepsilon>0$
there exists $C>0$ such that $K(n,\alpha,\varepsilon)\leq C$ for all $n\geq 1$.
We will show Ferenczi's result holds for any invariant measure by proving a stronger result.

We say that a subset $K$ of $X$ is an \emph{equicontinuous set in the mean}
if for every $\ep>0$,
there exists a $\delta>0$ such that for any $n\in \N$ and
all $x,y\in K$ with $d(x,y)<\delta$ we have $\frac{1}{n}\sum_{i=0}^{n-1} d(T^ix,T^iy)<\ep$.
$(X,T)$ is  \emph{$\mu$-equicontinuous in the mean} if for every $\tau>0$
there exists a subset $K$ of $X$ which is an equicontinuous set in the mean with $\mu(K)>1-\tau$.

\begin{thm}\label{0624}
Let $(X,T)$ be a t.d.s. and $\mu\in M(X,T)$. Then the following statements are equivalent:
\begin{enumerate}

\item $(X,T)$ is $\mu$-mean equicontinuous.

\item $(X,T)$ is $\mu$-equicontinuous in the mean.

\item For any finite partition $\alpha=\{A_1,\dotsc,A_l\}$ of $X$ and $\varepsilon>0$
there exists $C>0$ such that $K(n,\alpha,\varepsilon)\leq C$ for all $n\geq 1$.

\item For any finite partition $\alpha=\{A_1,\dotsc,A_l\}$ of $X$ and $\varepsilon>0$ there is $k\in \N$ and measurable sets ${B_1,B_2,\cdots,B_k}$
satisfying $\mu(\overset{k}{\underset{ i=1}{\cup}} B_i)>1-\varepsilon$, and if $x,y\in B_i$ for some $1\leq i\leq k$, then
$\underset{n\ra \infty}{\limsup}  H_n^\alpha(x,y)<\varepsilon$.

\item For any finite partition $\alpha=\{A_1,\dotsc,A_l\}$ of $X$ and $\varepsilon>0$ there is $k\in \N$ and measurable sets ${B_1,B_2,\cdots,B_k}$
satisfying $\mu(\overset{k}{\underset{ i=1}{\cup}} B_i)>1-\varepsilon$, and if $x,y\in B_i$ for some $1\leq i\leq k$, then
$H_n^\alpha(x,y)<\varepsilon$ for any $n\in \N$.
\end{enumerate}
\end{thm}

\begin{proof}$(1)\Leftrightarrow(2)$ is \cite[Theorem 4.2]{HLTXY}.

$(2)\Rightarrow(3)$
Let $\alpha=\{A_1,\dotsc,A_l\}$ be a measurable partition of $X$.
By the regularity of $\mu$, $A_i$ can be approximated in
measure by a sequence of its closed subsets $P_i^m$ such that $\mu(A_i\setminus P_i^m)< \frac{1}{m}\mu(A_i)$.
We let $U_m=X\setminus \overset{l}{\underset{i=1}{\cup}}P_i^m$. Then $\mu(U_m)<\frac{1}{m}$ for $m\in \N$.
Let $s_m(i,j)=\min\{d(x_i^m,x_j^m):x_i^m\in P_i^m,x_j^m\in P_j^m\}$ and $s_m=\underset{1 \leq i\not=j\leq l}{\min} s_m(i,j)$.
Fix $\ep>0$, there is $m'\in \N$ such that $\mu(U_{m'})<\frac{1}{m'}<\frac{\ep^2}{2}$.

As $(X,T)$ is  $\mu$-equicontinuous in the mean,
There is a subset $K$ with $\mu(K)>1-\frac{\ep^2}{2}$
and $\delta>0$ such that
\[
 \frac{1}{n}\sum_{i=0}^{n-1}d(T^ix,T^iy)<\ep^2 s_{m'}
\]
for all $n\in \N$ and $x,y\in K$ with $d(x,y)<\delta$.
Let $U'=K\cap U_{m'}^c$, choose compact set $U\subset U'$ such that $\mu(U)>1-\ep^2$.

For $n\in\mathbb{N}$ and $x\in X$, let $E(x)=\{i\ge 0: T^ix\in U\}$ and $ E_n=\{x\in X:\frac{\#(E(x)\cap [0,n-1])}{n}> 1-\epsilon\}.$
Then $\mu(E_n)> 1-\epsilon$. By the regularity of $\mu$, we can choose compact set $G_n\subset E_n\cap U$ such that $\mu(G_n)>1-\ep-\ep^2$.

For $x,y\in G_n$ with $d(x,y)<\delta$,  let  $A_n(x,y)=\{0\leq i \leq n-1: d(T^ix,T^iy)<  s_{m'} \},$
then $\frac{1}{n}\#(A_n(x,y))\geq 1-\ep^2$.
For $i\in E(x)\cap E(y)\cap A_n(x,y)$,  $T^ix,T^iy\in K\cap (\overset{l}{\underset{j=1}{\cup}}P_j^{m'})$.
By the definition of $s_{m'}$, we have
$(T^ix,T^iy)\in \overset{l}{\underset{j=1}{\cup}} P_j^{m'}\times P_j^{m'} \subset  \overset{l}{\underset{j=1}{\cup}}A_j\times A_j$.
So $H_n^\alpha(x,y)\leq 2\ep+\ep^2$ for any $n\in\N$.

Since $U$ is compact, there exist $x_1,x_2,\cdots,x_M\in U$ such that
$\bigcup_{r=1}^M B(x_r,\frac{\delta}{2})\supseteq U$.
Let $I_n=\{ r\in [1,m]: B(x_r,\frac{\delta}{2})\cap G_n\neq \emptyset\}$.
For $r\in I_n$, we choose $y_r^n\in B(x_r,\frac{\delta}{2})\cap G_n$.
Then
$$\bigcup_{r\in I_n} \big( B(y_r^n,\delta)\cap G_n\big) \supseteq \bigcup_{r\in I_n} \big( B(x_r,\frac{\delta}{2})\cap G_n\big)=G_n.$$
So $\mu(\bigcup_{r\in I_n }B_n^\alpha(y_r^n,2\ep+\ep^2) )>1-\ep-\ep^2,$ then $K(n,\alpha,\ep^2+2\ep )\leq M$ for any $n\in \N$.

$(3)\Rightarrow(5)$
Let $\ep>0$. There is $C=C(\ep)$ such that for any $n\in\N$,
there is $F_n\in X$ with $\#(F_n)\leq C$  such that
\[
\mu\biggl(\bigcup_{x\in F_n} B_{n}^\alpha(x,\ep/8)\biggr)>1-\ep/8.
\]
Let $D=X\times X\setminus \overset{l}{\underset{ i=1}{\cup}} A_i\times A_i$.
By the Birkhoff pointwise ergodic theorem for $\mu\times \mu$ a.e. $(x,y)\in X^2$
$$H_n^\alpha(x,y)=\frac{1}{n}\sum_{i=0}^{n-1}1_D(T^ix,T^iy)\rightarrow H^*(x,y).$$
So for a given $0<r< \min\{1,\frac{\ep}{2C}\}$, by Egorov's theorem there are $R\subset X^2$ with $\mu\times \mu(R)>1-r^2$ and $N_0\in\N$
such that if $(x,y)\in R$ then
$$|H_n^\alpha(x,y)-H_{N_0}^\alpha(x,y)|<r,\ \text{for}\ n\ge N_0.$$
By Fubini's theorem there is $A\subset X$ such that $\mu(A)>1-r$ and for any $x\in A$,
$\mu(R_x)>1-r$, where
\[R_x=\{y\in X: (x,y)\in R\}.\]
Enumerate $F_{N_0}=\{x_1,x_2,\dotsc,x_m\}$. Then $m\leq C$.
Let $I=\{1\leq i\leq m\colon A\cap B_{N_0}^\alpha(x_i, \ep/8)
\neq\emptyset \}$.
Denote $\#(I)=m'$. Then $1\leq m'\leq m$.
For each $i\in I$, pick $y_i\in A\cap  B_{N_0}^\alpha(x_i, \ep/8)$.
Then we have $B_{N_0}^\alpha(x_i, \ep/8)\subset
B_{N_0}^\alpha(y_i, \ep/4)$ for all $i\in I$.
As
\begin{align*}
\mu\biggl( A \cap \bigcap_{i\in I} R_{y_i}
\cap \bigcup_{x\in F_{N_0}} B_{N_0}^\alpha(x, \ep/8)
\biggr)\geq 1-r-m'r-\ep/8>1-\ep,
\end{align*}
choose a compact subset
\[K \subset A \cap \bigcap_{i\in I} R_{y_i}
\cap \bigcup_{x\in F_{N_0}} B_{N_0}^\alpha(x, \ep/8)\] with $\mu(K)>1-\ep$.
If $x\in K$, there exists $i\in I$
such that $x\in R_{y_i}\cap B_{N_0}^\alpha(y_i, \ep/4)$.
Then $(y_i,x)\in R$.
By the construction of $R$,
for any $n\ge N_0$,
\[H_{n}^\alpha(x,y_i)\le
H_{N_0}^\alpha(x,y_i)+r<\ep/4+r<\ep/2.\]

Let $\mathcal{C}=\bigvee_{i=0}^{N_0-1} T^{-i}\alpha=\{C_1,C_2,\cdots,C_p\}$ with $p\leq l^{N_0}$. Let $E_j^i=C_j\cap B_{N_0}^\alpha(y_i, \ep/4)$.
Then $K\subset \underset{ i\in I}{\bigcup}\overset{p}{\underset{ j=1}{\cup}} E_j^i$.
For any $(x,y)\in E_j^i$, we have $H_{n}^\alpha(x,y)=0$ for any $1\leq n\leq N_0-1$, thus $H_{n}^\alpha(x,y)<\ep$ for any $n\in \N.$

$(5)\Rightarrow(2)$  Assume that diam$(X)\leq 1$.
For every $\varepsilon>0$ and any $m\in \N$ there is  finite partition $\alpha^m=\{A_m^1,A_m^2,\cdots,A_m^{t_m}\}$ of $X$
with $diam(A_m^i)<\frac{\ep}{2^m}$ such that
for any $b\in \N$
there is $n_{m,b}\in \N$ and measurable sets ${B_{m,b}^1,B_{m,b}^2,\cdots,B_{m,b}^{n_{m,b}}}$
 satisfying $\mu(\overset{n_{m,b}}{\underset{ i=1}{\cup}} B_{m,b}^i)>1-\frac{\varepsilon}{2^b}\frac{1}{2^m}$,
 and if $x,y\in B_{m,b}^i$ for $1\leq i\leq n_{m,b}$ then
 $ H_n^{\alpha^m}(x,y)<\frac{\varepsilon}{2^b}$ for any $n\geq 1$. By the regularity of $\mu$, we can assume that $B_{m,b}^i$  is a closed set.
 Let $K_{m,b}=  \overset{n_{m,b}}{\underset{ i=1}{\cup}} B_{m,b}^i $, then  $\mu(K_{m,b})>1-\frac{\varepsilon}{2^b}\frac{1}{2^m}.$
 Let $K_m=\overset{\infty}{\underset{ b=1}{\cap}} K_{m,b}$, then  $\mu(K_{m})>1-\frac{\varepsilon}{2^m}.$
 Let $K=\overset{\infty}{\underset{ m=1}{\cap}} K_m$, then  $\mu(K)>1-\varepsilon.$

 Now we show $K$ is equicontinuous in the mean.
 Suppose to the contrary that $K$ is not equicontinuous in the mean.
 Then by the definition there exists $\eta>0$ such that for any $k\geq 1$ there are $x_k,y_k\in K$
 and $s_k$ such that $d(x_k,y_k)<\frac{1}{k}$ and $\bar d_{s_k}(x_k,y_k)\geq \eta$.
 As $K$ is compact, without loss of generality assume that $x_k\to x\in K$,
 we also have $y_k\to x$.
 For any $k\in\N$, by the triangle inequality, either
 $\bar d_{s_k}(x_k,x)\geq \frac{\eta}{2}$ or $\bar d_{s_k}(y_k,x)\geq \frac{\eta}{2}$.
 Without loss of generality, we always have $\bar d_{s_k}(x_k,x)\geq \frac{\eta}{2}$.

 Let $E_m^k=\{0\leq i\leq s_k-1\colon \alpha_i^m(x)\neq \alpha_i^m(x_k)\}.$ Then $H^{\alpha^m}_{s_k}(x_k,x)=\frac{1}{s_k}\#(E_m^k)$.
 Since  diam$(X)\leq 1$ and diam$(A_m^i)<\frac{\ep}{2^m}$.
 So if $i\in E_m^k$, then $d(x,x_k)\leq 1$, if $i \in [0,s_k-1]\setminus E_m^k$, then $d(x,x_k)\leq \frac{\ep}{2^m}$. So
 $$\frac{1}{s_k}(\#(E_m^k)\times 1+(s_k-\#(E_m^k))\times \frac{\ep}{2^m})\geq \bar d_{s_k}(x_k,x)\geq \frac{\eta}{2}.$$
 Fix $m'\in \N$ with $\frac{\ep}{2^{m'}}\leq \frac{\eta}{4}$, we have $\#(E_{m'}^k)\geq \frac{\eta}{4}s_k$.
 Then $H^{\alpha^{m'}}_{s_k}(x_k,x)\geq \frac{\eta}{4}$.
 Choose $b'\in \N$ such that $\frac{\ep}{2^{b'}}< \frac{\eta}{4}$.
 Without loss of generality,  we assume $\{x_k\}\subset B_{m',b'}^i$ for some $1\leq i\leq n_{m',b'}$.
 Since $B_{m',b'}^i$ is closed, $x\in B_{m',b'}^i$,
 then $ H^{\alpha^{m'}}_{s_k}(x_k,x)< \frac{\ep}{2^{b'}}<\frac{\eta}{4} $, a contradiction.

$(5)\Rightarrow(3)$ and $(5)\Rightarrow(4)$ are obvious.

$(4)\Rightarrow(5)$
Fix $\ep>0$, then there exist measurable sets ${B_1,B_2,\cdots,B_k}$
satisfying $\mu(\overset{k}{\underset{ i=1}{\cup}} B_i)>1-\frac{\varepsilon}{8}$ and
 $\underset{n\ra \infty}{\limsup} H_n^\alpha(x,y)<\varepsilon/8$
for $x,y\in \overset{k}{\underset{ i=1}{\cup}} B_i \times B_i$. Let $D=X\times X\setminus \overset{l}{\underset{ i=1}{\cup}} A_i\times A_i$.
By the Birkhoff pointwise ergodic theorem for $\mu\times \mu$ a.e. $(x,y)\in X^2$
$$H_n^\alpha(x,y)=\frac{1}{n}\sum_{i=0}^{n-1}1_D(T^ix,T^iy)\rightarrow H^*(x,y).$$
So for a given $0<r< \min\{1,\frac{\ep}{2k}\}$, by Egorov's theorem there are $R\subset X^2$ with $\mu\times \mu(R)>1-r^2$ and $N_0\in\N$
such that if $(x,y)\in R$ then
$$|H_n^\alpha(x,y)-H^*(x,y)|<r,\ \text{for}\ n\ge N_0.$$
By Fubini's theorem there is $A\subset X$ such that $\mu(A)>1-r$ and for any $x\in A$,
$\mu(R_x)>1-r$, where
\[R_x=\{y\in X: (x,y)\in R\}.\]

Let  $B^*(x,\delta)=\{y\in X\colon H^*(x,y)<\delta\}$.
Choose $x_i\in B_i$ and
let $I=\{1\leq i\leq k\colon A\cap B^*(x_i, \ep/8)
\neq\emptyset \}$.
Denote $\#(I)=k'$. Then $1\leq k'\leq k$.
For each $i\in I$, pick $y_i\in A\cap  B^*(x_i, \ep/8)$.
Then we have $B^*(x_i, \ep/8)\subset
B^*(y_i, \ep/4)$ for all $i\in I$.
As
\begin{align*}
\mu\biggl( A \cap \bigcap_{i\in I} R_{y_i}
\cap \bigcup_{i=1}^k B^*(x_i, \ep/8)
\biggr)\geq 1-r-m'r-\ep/8>1-\ep,
\end{align*}
choose a compact subset
\[K \subset A \cap \bigcap_{i\in I} R_{y_i}
\cap \bigcup_{i=1}^k B^*(x_i, \ep/8)\] with $\mu(K)>1-\ep$.
If $x\in K$, there exists $i\in I$
such that $x\in R_{y_i}\cap B^*(y_i, \ep/4)$.
Then $(y_i,x)\in R$.
By the construct of $R$,
for any $n\ge N_0$,
\[H_{n}^\alpha(x,y_i)\le
H^*(x,y_i)+r<\ep/4+r<\ep/2.\]

Let $\mathcal{C}=\bigvee_{i=0}^{N_0-1} T^{-i}\alpha=\{C_1,C_2,\cdots,C_p\}$ with $p\leq l^{N_0}$. Let $E_j^i=C_j\cap B_{N_0}^\alpha(y_i, \ep/4)$.
Then $K\subset \underset{ i\in I}{\bigcup}\overset{p}{\underset{ j=1}{\cup}} E_j^i$.
For any $(x,y)\in E_j^i$, we have $H_{n}^\alpha(x,y)=0$ for any $n\leq N_0-1$, thus $H_{n}^\alpha(x,y)<\ep$ for any $n\in \N.$

\end{proof}

\begin{cor}
	Let $(X,\B,\mu,T)$ be a m.p.s. where $(X,\mathcal{B},\mu)$ is a Lebesgue space.
Then the following statements are equivalent:
\begin{enumerate}

\item $(X,\B,\mu,T)$ has discrete spectrum.

\item For any finite partition $\alpha=\{A_1,\dotsc,A_l\}$ of $X$ and $\varepsilon>0$
there exists $C>0$ such that $K(n,\alpha,\varepsilon)\leq C$ for all $n\geq 1$.

\end{enumerate}
\end{cor}
\begin{proof}  By Theorem \ref{0625}, there is a topological dynamical system $(Y,S)$ such that $\pi:(Y,\B_Y,\nu,S)\ra (X,\B,\mu,T)$ is measure-theoretically isomorphic, where $\nu$ is an $S$-invariant Borel probability measure and $\B_Y$ is the $\sigma$-algebra
consisting of Borel subsets of $Y$.

$(1)\Rightarrow(2)$ Then $(Y,\B_Y,\nu,S)$ has discrete spectrum. By \cite[Theorem 4.2, 4.3, 4.6]{HLTXY}, $(Y,\B_Y,\nu,S)$ is $\mu$-mean equicontinuous.
By Theorem \ref{0624} for any finite partition $\pi^{-1}\alpha=\{\pi^{-1}A_1,\dotsc,\pi^{-1}A_l\}$ on $(Y,\B_Y,\nu,S)$ and $\varepsilon>0$
there exists $C>0$ such that $K(n,\pi^{-1}\alpha,\varepsilon)\leq C$ for all $n\geq 1$. Then for any finite partition $\alpha=\{ A_1,\dotsc, A_l\}$ on $(X,\B,\mu,T)$ and $\varepsilon>0$
there exists $C>0$ such that $K(n,\alpha,\varepsilon)\leq C$ for all $n\geq 1$.

$(2)\Rightarrow(1)$ just the same as $(1)\Rightarrow(2)$.
\end{proof}

\section{$\mu$-mean equicontinuous functions and complexity functions}

\subsection{$\mu$-mean equicontinuous functions}
In this subsection, we will give two definitions about measure-theoretic mean equicontinuous functions, and show that they are equivalent.
\begin{defn}
Let $(X,T)$ be a t.d.s. with a metric $d$ and $\mu\in M(X,T)$. We say that a function $f\in L^1(X,\mu)$ is
 $(\mu,d)$-mean equicontinuous if for any $\tau>0$ there is a compact subset $K$ of $X$ satisfying
$\mu(K)>1-\tau$, and for any $\ep$ there is $\delta>0$ whenever $d(x,y)<\delta$ with $x,y\in K$, we have
$$ \limsup_{n\to\infty}\frac{1}{n}\sum_{i=0}^{n-1}|f(T^ix)-f(T^iy)|<\ep.$$

\begin{defn}
Let $(X,\B,\mu,T)$ be a measure preserving system.
We say a function $f\in L^1(X,\mu)$ is
 $\mu$-mean equicontinuous if for any $\ep>0$ there is $k\in \N$ and measurable sets ${A_1,A_2,\cdots,A_k}$
satisfying $\mu(\overset{k}{\underset{ i=1}{\cup}} A_i)>1-\ep$, and if $x,y\in A_i$ for some $1\leq i\leq k$ then
$$ \limsup_{n\to\infty}\frac{1}{n}\sum_{i=0}^{n-1}|f(T^ix)-f(T^iy)|<\ep.$$
\end{defn}

\begin{prop}\label{0620}
Let $(X,T)$ be a t.d.s. with a metric $d$, $\mu\in M(X,T)$ and $f\in L^1(X,\mu)$. Then the following statements are equivalent:
\begin{enumerate}

\item $f$ is $(\mu,d)$-mean equicontinuous.

\item $f$ is $\mu$-mean equicontinuous.

\end{enumerate}

\end{prop}
\begin{proof}$(1)\Rightarrow(2)$ For any $\ep>0$, there is a compact set $M$ with $\mu(M)>1-\ep$,
and there is $\delta>0$ whenever $d(x,y)<2\delta$ with $x,y\in M$
$$ \limsup_{n\to\infty}\frac{1}{n}\sum_{i=0}^{n-1}|f(T^ix)-f(T^iy)|<\ep.$$
By the compactness of $M$, there is $k\in \N$ and $x_1,\dots, x_k\in M$ such that
$M=\overset{k}{\underset{ i=1}{\cup}}B(x_i,\delta)\cap M.$
Let  $A_i=B(x_i,\delta)\cap M$ for $1\leq i\leq k$, then $\mu(\overset{k}{\underset{ i=1}{\cup}} A_i)>1-\ep$, thus
 $f$ is $\mu$-mean equicontinuous.

$(2)\Rightarrow(1)$ For any $\tau>0$ and any $n\in \N$, there exist $k_n$ pairwise disjoint measurable subsets $\{A_i^n\}_{i=1}^{k_n}$ with
$\mu(\overset{k_n}{\underset{ i=1}{\cup}} A_i^n)>1-\frac{\tau}{2^{n}}$ such that if $x,y \in A_i^n$,
$$ \limsup_{m\to\infty}\frac{1}{m}\sum_{i=0}^{m-1}|f(T^ix)-f(T^iy)|<\frac{1}{n}.$$
By the regularity of $\mu$, we can require $A_i^n$ to be compact.
Let $\delta_n=\underset{ 1\leq i\not=j \leq k_n}{\min}d(A_i^n,A_j^n).$
Let $M_n=\overset{k_n}{\underset{ i=1}{\cup}} A_i^n$, $M=\overset{\infty}{\underset{n=1}{\cap}} M_n$, then $\mu(M)>1-\tau$ and $M$ is a compact set.
For any $\ep>0$ there is $n\in \N$ such that $\frac{1}{n}<\ep$. Then whenever $d(x,y)<\delta_n$ with $x,y\in M$, there exists $A_{i}^n$ such that
$x,y\in A_{i}^n$. Then $$ \limsup_{n\to\infty}\frac{1}{m}\sum_{i=0}^{m-1}|f(T^ix)-f(T^iy)|<\frac{1}{n}<\ep.$$ Then $f$ is $(\mu,d)$-mean equicontinuous.
\end{proof}

\begin{rem}This Proposition show that for a t.d.s. $(X,T)$, $\mu\in M(X,T)$,
the definition of $(\mu,d)$-mean equicontinuous functions does not rely on the metric $d$.
Thus in the following, we will simply write $\mu$-mean equicontinuous functions instead of $(\mu,d)$-mean equicontinuous functions.
\end{rem}

\begin{defn}
Let $(X,T)$ be a t.d.s. with a metric $d$ and $\mu\in M(X,T)$. We say a function $f\in L^1(X,\mu)$ is
  $\mu$-equicontinuous in the mean if for any $\tau>0$ there is a compact subset $K$ of $X$ satisfying
$\mu(K)>1-\tau$, and for any $\ep$ there is $\delta>0$ whenever $d(x,y)<\delta$ with $x,y\in K$, we have
$$\frac{1}{n}\sum_{i=0}^{n-1}|f(T^ix)-f(T^iy)|<\ep,$$ for any $n\ge 1$.
\end{defn}

\begin{thm}\label{0621}
Let $(X,T)$ be a t.d.s. with a metric $d$, $\mu\in M(X,T)$ and $f\in L^1(X,\mu)$.
Then the following statements are equivalent:
\begin{enumerate}
	\item  $f$ is $\mu$-equicontinuous in the mean.
	\item $f$ is $\mu$-mean equicontinuous.
\end{enumerate}
\end{thm}
\begin{proof}
(1) $\Rightarrow$ (2) is obvious.

\medskip
(2) $\Rightarrow (1)$ Now assume that $f$ is $\mu$-mean equicontinuous.
Fix $\varepsilon>0$. By Lusin's theorem, there is a compact $K\subset X$ such that $\mu(K)>1-\ep$ and $K$ is an $f$-mean equicontinuous set and
$f$ is continuous on $K$.
For any $l >0$ there exists a $\delta_l>0$ such that
\[
\limsup_{n\ra \infty} \frac{1}{n}\sum_{i=0}^{n-1}|f(T^ix)-f(T^iy)|<\frac{1}{l}
\]
for all $x,y\in K$ with $d(x,y)<\delta_l$.
As $K$ is compact,
there exists a finite subset $F_l$ of $K$ such that
$K\subset \bigcup_{x\in F_l}B(x,\delta_l)$.
Enumerate $F_l$ as $\{x_l^1,x_l^2,\dotsc,x_l^{m_l}\}$.
For $j=1,\dotsc,m_l$ and $N\in\N$, let
\[A_l^N(x_l^j)=\biggl\{y\in B(x_l^j,\delta_l)\cap K\colon
\frac{1}{n}\sum_{i=0}^{n-1}|f(T^ix_l^j)-f(T^iy)|<\frac{1}{l},\
n=N,N+1,\dotsc\biggr\}.
\]
It is easy to see that for each $j=1,\dotsc,m_l$,
$\{A_l^N(x_j)\}_{N=1}^\infty$ is an increasing sequence and $B(x_l^j,\delta_l)\cap K=\underset{N=1}{\overset{\infty}{\cup}} A_l^N(x_l^j)$.
Choose $N_l\in\N$ and a compact subset $K_l$ of $K$
such that $\mu(K_l)>\mu(K)-\frac{\ep}{2^l}$ with $K_l\cap B(x_l^j,\delta_l)\subset A_l^{N_l}(x_l^j)$ for $j=1,\dotsc,m_l$.

Let $K_0=\underset{j=1}{\overset{\infty}{\cap}} K_j$, then $\mu(K_0)>1-2\ep$. Let $\delta_{l,1}$ be the Lebesgue number of $\{B(x_l^j,\delta_l):j=1,\cdots,m_l\}.$
By the continuity of $f$ on $K_0$,
there exists $\delta_{l,2}>0$ such that $$\frac{1}{n}\sum_{i=0}^{n-1}|f(T^ix)-f(T^iy)|<\frac{1}{l}$$ for $n=1,2,\cdots,N_l-1$
and every $x,y\in K_0$ with $d(x,y)<\delta_{l,2}$.
Let $\delta_{l,3}=\min\{\delta_{l,1},\delta_{l,2}\}$.
Then by the construction, we know that for every $x,y\in K_0$ with $d(x,y)<\delta_{l,3}$, and every $n\in \N$,
we have $\frac{1}{n}\sum_{i=0}^{n-1}|f(T^ix)-f(T^iy)|<\frac{2}{l}$. Thus $f$ is $\mu$-equicontinuous in the mean.

\end{proof}

\subsection{Bounded complexity with respect to functions}
In this subsection, we study $\mu$-mean equicontinuous functions and  measure-theoretic complexity with respect to some functions.

Let $(X,T)$ be a t.d.s.\ and
$\mu\in M(X,T)$, $f\in L^1(X,\mu)$. Define
\[ \bar{f}_n(x,y)= \frac{1}{n}\overset{n-1}{\underset{ i=0}{\sum}}|f(T^ix)-f(T^iy)| \text{ and }
\hat{f}_n(x,y)= \max\Bigl\{\bar f_k(x,y)\colon 1\le k\le n\Bigr\}.\]

For $x\in X$, $n\geq 1$ and $\varepsilon>0$, let
\[ B_{\bar{f}_n}(x,\varepsilon)=\{y\in X\colon \bar{f}_n(x,y)<\varepsilon\}.\]
\[ B_{\hat{f}_n}(x,\varepsilon)=\{y\in X\colon \hat{f}_n(x,y)<\varepsilon\}.\]
and
\[
K(n ,\varepsilon, {\bar{f}} )=\min\{\#(F)  \colon \mu(\bigcup_{x\in F }B_{\bar{f}_n}(x,\varepsilon) )>1-\varepsilon\}.
\]
\[
K(n ,\varepsilon, {\hat{f}} )=\min\{\#(F)  \colon \mu(\bigcup_{x\in F }B_{\hat{f}_n}(x,\varepsilon) )>1-\varepsilon\}.
\]

We say that $\mu$   has \emph{bounded topological complexity
with respect to $\{\bar{f}_n\}$} (resp. $\{\hat{f}_n\}$)
if for every $\ep>0$ there exists a positive integer $C=C(\ep)$
such that $K(n ,\varepsilon, {\bar{f}} )\leq C$ (resp. $K(n ,\varepsilon, {\hat{f}} )\leq C$) for all $n\geq 1$.
Motivated by \cite[Theorem 4.2, 4.3]{HLTXY}, we have the following theorem.

\begin{thm}\label{0622}
Let $(X,T)$ be a t.d.s.\ and
$\mu\in M(X,T)$, $f\in L^1(X,\mu)$.
Then the following statements are equivalent:
\begin{enumerate}
\item $f$ is $\mu$-mean equicontinuous.
\item $f$ is $\mu$-equicontinuous in the mean.

\item $\mu$ has bounded  complexity
with respect to   $\{\bar{f}_n\}$.

\item $\mu$ has bounded  complexity
with respect to   $\{\hat{f}_n\}$.

\item For any $\ep>0$ there is $k\in \N$  and measurable subsets $A_1,\cdots,A_k$ with $\mu(\overset{k}{\underset{ i=1}{\cup}} A_i)>1-\ep$ and
  $\hat f_n(x,y)<\ep$ for any $x,y\in A_i$ and any $n\in \N$.
\end{enumerate}

\end{thm}

\begin{proof}
$(1)\Leftrightarrow(2)$ is Theorem \ref{0621}.

$(2)\Rightarrow(5)$
Fix $\varepsilon>0$,
there exists a  measurable subset $K$ of $X$ with $\mu(K)>1-\varepsilon$ such that $K$ is $f$-equicontinuous in the mean.
As the measure $\mu$ is regular, we can require  $K$ is compact.
There exists a $\delta>0$ such that for all $x,y\in K$ and any $n\in \N$ with $d(x,y)<\delta$
\[
\bar{f_n}(x,y)= \frac{1}{n}\sum_{i=0}^{n-1}|f(T^ix)-f(T^iy)|<\ep,
\]
thus $\hat{f_n}(x,y)<\ep$ for any $n\in \N$.
As $K$ is compact,
there exists a finite subset $F$ of $K$ such that
$K\subset \bigcup_{x\in F}B(x,\frac{\delta}{2})$.
Enumerate $F$ as $\{x_1,x_2,\dotsc,x_m\}$.
Then $\hat{f_n}(x,y)<\ep$ for any $n\geq 1$ with $x,y\in B(x_i,\frac{\delta}{2})$, $1\leq i\leq m$.

$(5)\Rightarrow(2)$
For any $\ep>0$ and any $m\in \N$ there is subset $K_m\subset X$ with $\mu(K_m)>1-\frac{\ep}{2^m}$ and
 $k_m\in \N$ such that $K_m\subset \overset{k_m}{\underset{ i=1}{\cup}} A_m^i$ such that $\hat f_n(x,y)<\frac{\ep}{2^m}$ for any $x,y\in A_m^i$ and any $n\in \N$.
 By the regularity of $\mu$, we can assume $A_m^i$ and $K_m$ to be compact.
 Let $K=\overset{\infty}{\underset{ m=1}{\cap}} K_m$, then $\mu(K)>1-\ep$.

Now we show $K$ is $f$-equicontinuous in the mean.
Assume contrary that $K$ is not $f$ equicontinuous in the mean.
Then by the definition there exists $\tau>0$ such that for any $k\geq 1$ there are $x_k,y_k\in K$ and $m_k\in \N$
 such that $d(x_k,y_k)<\frac{1}{k}$ and $\hat f_{m_k}(x_k,y_k)\geq \tau$.
As $K$ is compact, without loss of generality assume that $x_k\to x\in K$,
we also have $y_k\to x$.
For any $k\in\N$, by the triangle inequality, either
$\hat f_{m_k}(x_k,x)\geq \frac{\tau}{2}$ or $\hat f_{m_k}(y_k,x)\geq \frac{\tau}{2}$.
Without loss of generality, we always have $\hat f_{m_k}(x_k,x)\geq \frac{\tau}{2}$.

Choose $m'\in \N$ such that $\frac{\ep}{2^{m'}}<\frac{\tau}{2}$.
Note that the sequence $\{x_k\}$ is in $K$ and $\{A_{m'}^i\}$ is a finite cover of $K$.
Passing to a subsequence if necessary, we assume that the sequence $\{x_k\}$ is in $A_{m'}^{i_0}$.
As $A_{m'}^{i_0}$ is closed, $x$ is also in $A_{m'}^{i_0}$.
Note that for any $u,v\in A_{m'}^{i_0}$ and any $n\geq 1$,
$\hat f_n(u,v)<\frac{\ep}{2^{m'}}<\frac{\tau}{2}$. Particularly, we have
$\hat f_{m_k}(x_k,x)<\frac{\ep}{2^{m'}}<\frac{\tau}{2}$, a contradiction.

$(5)\Rightarrow(4)$ and $(4)\Rightarrow(3)$ are obvious.

$(3)\Rightarrow(5)$
Assume that $\mu$ has bounded  complexity
with respect to  $\{\bar{f}_n\}$.
Let $\ep>0$. There is $C=C(\ep)$ such that for any $n\in\N$,
there is $F_n\in X$ with $\#(F_n)\leq C$  such that
\[
\mu\biggl(\bigcup_{x\in F_n} B_{\bar{f}_n}(x,\ep/8)\biggr)>1-\ep/8.
\]
By the Birkhoff pointwise ergodic theorem for $\mu\times \mu$ a.e. $(x,y)\in X^2$
$$\bar f_n(x,y)=\frac{1}{n}\sum_{i=0}^{n-1}|f(T^ix)-f(T^iy)|\rightarrow f^*(x,y).$$
So for a given $0<r< \min\{1,\frac{\ep}{2C}\}$, by Egorov's theorem there are $R\subset X^2$ with $\mu\times \mu(R)>1-r^2$ and $N_0\in\N$
such that if $(x,y)\in R$ then
$$|\bar f_n(x,y)-\bar f_{N_0}(x,y)|<r,\ \text{for}\ n\ge N_0.$$
By Fubini's theorem there is $A\subset X$ such that $\mu(A)>1-r$ and for any $x\in A$,
$\mu(R_x)>1-r$, where
\[R_x=\{y\in X: (x,y)\in R\}.\]
Enumerate $F_{N_0}=\{x_1,x_2,\dotsc,x_m\}$. Then $m\leq C$.
Let $I=\{1\leq i\leq m\colon A\cap B_{\bar{f}_{N_0}}(x_i, \ep/8)
\neq\emptyset \}$.
Denote $\#(I)=m'$. Then $1\leq m'\leq m$.
For each $i\in I$, pick $y_i\in A\cap  B_{\bar{f}_{N_0}}(x_i, \ep/8)$.
Then we have $B_{\bar{f}_{N_0}}(x_i, \ep/8)\subset
B_{\bar{f}_{N_0}}(y_i, \ep/4)$ for all $i\in I$.
As
\begin{align*}
\mu\biggl( A \cap \bigcap_{i\in I} R_{y_i}
\cap \bigcup_{x\in F_{N_0}} B_{\bar{f}_{N_0}}(x, \ep/8)
\biggr)\geq 1-r-m'r-\ep/8>1-\ep,
\end{align*}
for any $i\in I$, by Lusin's theorem, we can choose compact subset
\[K_i \subset A \cap \bigcap_{i\in I} R_{y_i}
\cap B_{\bar{f}_{N_0}}(x_i, \ep/8)\] with $\mu(\underset{i\in I}{\cup} K_i)>1-\ep$ and $f(T^jx)$ is continuous on $\underset{i\in I}{\cup} K_i$ for $0\leq j \leq N_0$.
Let $K=\underset{i\in I}{\cup} K_i$.

If $x\in K$, there exists $i\in I$
such that $x\in K_i\subset R_{y_i}\cap B_{\bar{f}_{N_0}}(y_i, \ep/4)$.
Then $(y_i,x)\in R$.
By the construct of $R$,
for any $n\ge N_0$,
\[\bar f_n(y_i,x)\le
\bar f_{N_0}(y_i,x)+r<\ep/4+r<\ep/2.\]

Since $f(T^jx)$ is continuous on $K$ for $0\leq j \leq N_0$, there is $\delta>0$ such that for any $x_1,x_2$ with $d(x_1,x_2)<\delta$,
we have $|f(T^jx_1)-f(T^jx_2)|<\ep$ for any $0\leq j \leq N_0$. Thus $\bar f_n(x_1,x_2)<\ep$ for $1\leq n \leq N_0$.
Since $K_i$ is compact, there exists a finite subset $H_i$
of $K_i$ such that $K_i\subset \underset{x\in H_i}{\bigcup}B(x,\frac{\delta}{2})$.
Then for any $x_1,x_2\in B(x,\frac{\delta}{2})\cap K_i$ where $x\in H_i$,
we have $\bar f_n(x_1,x_2)\leq \bar f_n(x_1,y_i)+\bar f_n(x_2,y_i)<\ep$ for any $n\geq N_0$, and $\bar f_n(x_1,x_2)< \ep$ for $1\leq n \leq N_0$.
Thus $\hat f_n(x_1,x_2)<\ep$ for any $n\geq 1$.
Thus $K\subset \underset{i\in I}{\bigcup} (\underset{x\in H_i}{\bigcup} B(x,\frac{\delta}{2})\cap K_i)$, and
for any $x_1,x_2\in B(x,\frac{\delta}{2})\cap K_i$ where $x\in H_i$, $\hat f_n(x_1,x_2)<\ep$ for any $n\geq 1$.
\end{proof}

\section{$\mu$-mean equicontinuous functions and almost periodic function}

Garc\'ia-Ramos and Marcus \cite[Corollary 3.8]{GM15} proved the following proposition holds for ergodic measure, we will show that it holds for invariant measure.

\begin{prop}\label{0515}\cite{GM15} Let $(X,T)$ be a t.d.s.,
$\mu \in M^e(X,T)$ and $f\in L^2(X,\mu)$. Then the following statements are equivalent:
\begin{enumerate}

\item $f$ is an almost periodic function in $L^2(X,\mu)$.

\item $f$ is $\mu$-mean equicontinuous.

\item $f$ is $\mu$-mean expansive, that is, there is $\delta>0$ such that \[
\mu\times\mu\biggl(\biggl\{(x,y)\in X\times X\colon
\liminf_{n\to\infty}\frac{1}{n}\sum_{i=0}^{n-1} |f(T^ix)-f(T^iy)|>\delta
\biggr\}\biggr)=1.\]
\end{enumerate}
\end{prop}

We use the same idea in \cite[Theorem 4.6]{HLTXY} and \cite[Proposition 4.1]{HWY} to prove the following theorem.

\begin{thm}\label{0623}
Let $(X,T)$ be a t.d.s.,
$\mu\in M(X,T)$ and $f\in L^2(X,\mu)$. Then the following statements are equivalent:
\begin{enumerate}

\item $f$ is an almost periodic function in $L^2(X,\mu)$.

\item $f$ is $\mu$-mean equicontinuous.

\end{enumerate}
\end{thm}

\begin{proof}
$(1)\Rightarrow(2)$  We will show that $\mu$ has bounded  complexity with respect to   $\{\bar{f}_n\}$, then by Theorem \ref{0622}, we know that $f$ is $\mu$-mean equicontinuous.

Put $\ep>0$, we will prove  there is $m\in \N$ such that $K(n ,\varepsilon, {\bar{f}} )\leq m$ for any $n\in \N$.
Let $\ep_t \ra 0$ when $t\ra \infty $ and $\overset{\infty}{\underset{ t=1}{\sum}} \ep_t^2<\ep^2/2$.
Since the collection of almost periodic function is spanned by the set of eigenfunctions of $L^2(X,\mu)$,
we can find eigenfunctions $\{ h_i\}_{i=1}^\infty$ of $L^2(X,\mu)$
and $a_{t,k}\in \mathbb{C}$ for $k=1,2,\cdots,K_t$ such that
\begin{align}\label{markov-ineq}
\|f-\sum_{k=1}^{K_t} a_{t,k} h_k\|_{L^2(X,\mu)}<\epsilon_t^2.
\end{align}

For $k\in \N$, there exist $\lambda_k \in \mathbb{C}$ with $|\lambda_k|=1$, and $X_k\in \mathcal{B}_X$ with $\mu(X_k)=1$,
$TX_k\subseteq X_k$ such that $h_k(Tx)=\lambda_k h_k(x)$ for all $x\in X_k$.
Put
$$Y_t=\{x\in X: |f(x)-\sum_{k=1}^{K_t} a_{t,k} h_k(x)|\ge \epsilon_t\}.$$
Then by  \eqref{markov-ineq}, $\mu(Y_t)<\epsilon_t^2$.
By Lusin's Theorem there exists a compact subset $A_t$ of $X\setminus Y_t$ such that $\mu(A_t)>1-\ep_t^2$ and  $h_k|_{A_t}$ is a continuous function
for $k=1,2,\cdots,K_t$. Clearly
\begin{align}\label{cha-ineq}
|f(x)-\sum_{k=1}^{K_t} a_{t,k} h_k(x)|< \ep_t
\end{align} for  $x\in A_t$.
By the continuity of $h_k|_{A_t}$, $k=1,2,\cdots,K_t$, there exists $\delta_t>0$ such that
when $x,y\in A_t$ with $d(x,y)<\delta_t$, one has
\begin{align}\label{cha-ineq1}
 \sum_{k=1}^{K_t} |a_{t,k}|\cdot |h_k(x)-h_k(y)|<\ep_t.
\end{align}

Set $X_0=(\overset{\infty}{\underset{ k=1}{\cap}} X_k)\cap (\overset{\infty}{\underset{ t=1}{\cap}} A_t)$. Then
$$\mu(X_0)>1-\overset{\infty}{\underset{ t=1}{\sum}} \epsilon_t^2>1-\frac{\ep^2}{2}.$$
By Lusin's theorem and the regularity of $\mu$ there is a compact set $A\subset X_0$ such that $\mu(A)>1-\frac{\ep^2}{2}$ and $f$ is continuous on $A$.

Since $f$ is continuous on $A$, there is $M\in \N$ such that $|f(x)|<M$ for $x\in A$.
Since $\ep_t \ra 0$ when $t\ra \infty $ and $\overset{\infty}{\underset{ t=1}{\sum}}  \ep_t^2<\ep^2/2$, there is $a\in \N$ such that $10M\ep_a<\ep$.
Since $A$ is compact, there exist $x_1,x_2,\cdots,x_m\in A$ such that
$\overset{m}{\underset{ r=1}{\cup}} B(x_r,\frac{\delta_a}{3})\supseteq A$.

 Next we show that $K(n ,\varepsilon, {\bar{f}} )=\min\#\{F\subset X  \colon \mu(\bigcup_{x\in F }B_{\bar{f}_n}(x,\varepsilon) )>1-\varepsilon\}\leq m$ for all  $n\in \mathbb{N}$.

To do this for  $x\in X$, let $E(x)=\{i\ge 0: T^ix\in A_a\}.$
Then for $n\in \mathbb{N}$, let
$$F_n=\{x\in X:\frac{\#(E(x)\cap [0,n-1])}{n}\le 1-\epsilon_a\}.$$

Then we have $\mu(F_n)<\ep_a$. Put $D_n=A\setminus F_n$. Then
$$\mu(D_n)>1-(\ep^2/2+\ep_a)>1-\epsilon$$
and for $z\in D_n$,
\begin{equation}\label{esti-ineq}
\frac{\#(E(z)\cap [0,n-1])}{n}> 1-\ep_a.
\end{equation}

Let $I_n=\{ r\in [1,m]: B_{\bar f_n}(x_r,\frac{\delta_a}{3})\cap D_n\neq \emptyset\}$.
For $r\in I_n$, we choose $y_r^n\in B_d(x_r,\frac{\delta_a}{3})\cap D_n$. Then
$$\bigcup_{r\in I_n} \big( B(y_r^n,\delta_a)\cap D_n\big) \supseteq \bigcup_{r\in I_n} \big( B(x_r,\frac{\delta_a}{3})\cap D_n\big)=D_n.$$
Fix $r\in I_n$, for any $x\in B(y_r^n,\delta_a)\cap D_n$,
\begin{align*}
&\hskip0.5cm \overline{f}_n(x,y_r^n)=\frac{1}{n}\sum_{i=0}^{n-1}|f(T^ix)-f(T^iy_r^n)|\\
&\le \frac{1}{n}\big(\sum_{i\in [0,n-1]\cap E(x)\cap E(y_r^n)}|f(T^ix)-f(T^iy_r^n)|
+2M\#\{[0,n-1]\setminus E(x)\}+2M\#\{[0,n-1]\setminus E(y_r^n)\}\big)\\
&< 4M\ep_a+\frac{1}{n}\sum_{i\in [0,n-1]\cap E(x)\cap E(y_r^n)}|f(T^ix)-f(T^iy_r^n)|
\end{align*}
where the last inequality follows from \eqref{esti-ineq}.

For $i\in [0,n-1]\cap E(x)\cap E(y_r^n)$,
\begin{align*}
&   |f(T^ix)-f(T^iy_r^n)|\leq|f(T^ix)-\sum_{k=1}^{K_t} a_{t,k} h_k(T^ix)|\\
&   +|f(T^iy_r^n)-\sum_{k=1}^{K_t} a_{t,k} h_k(T^iy_r^n)|+|\sum_{k=1}^{K_t} a_{t,k} (h_k(T^ix)-h_k(T^iy_r^n)|\\
&<2\ep_a+ |\sum_{k=1}^{K_t} a_{t,k} (h_k(T^ix)-h_k(T^iy_r^n)| \ \ \text{ (by \eqref{cha-ineq})}\\
&\le 2\ep_a+ \sum_{k=1}^{K_t} |a_{t,k}|\cdot|h_k(x)-h_k(y_r^n)|<3\ep_a.
\end{align*}
where the last inequality follows from \eqref{cha-ineq1}.

Combining the above two inequalities, one has
\begin{align*}
\overline{f}_n(x,y_r^n)< 4M\ep_a+\frac{1}{n}\sum_{i\in [0,n-1]\cap E(x)\cap E(y_r^n)}|f(T^ix)-f(T^iy_r^n)|\leq (4M+3)\ep_a
\le \epsilon.
\end{align*}
This implies $B(y_r^n,\delta_a)\cap D_n\subseteq B_{\overline{f}_n}(y_r^n,\epsilon)$ for $r\in I_n$.

Summing up $\bigcup_{r\in I_n} B_{\overline{f}_n}(y_r^n,\epsilon)\supseteq \bigcup_{r\in I_n}B(y_r^n,\delta_a)\cap D_n=D_n$ and
$\mu(D_n)>1-\epsilon$. Hence
$K(n ,\varepsilon, {\bar{f}} )\le |I_n|\le m$.

$(2)\Rightarrow(1)$
Let $G$ be the collection of points $z\in X$
which are generic to some ergodic measure,
that is,
for each $z\in G$, $\frac{1}{n}\sum_{i=0}^{n-1}\delta_{T^iz}\to \mu_z$
as $n\to\infty$
and $\mu_z$ is ergodic.
Then $G$ is measurable and $\mu(G)=1$.
We first prove the following Claim.

\medskip
\noindent\textbf{Claim 1}: $f$ is an almost periodic function in  $L^2(X,\mu_z)$ for $\mu$-a.e.\ $z\in G$.
\begin{proof}[Proof of the Claim 1]
Let $G_1=\{z\in G\colon$ $f\in  L^2(X,\mu_z)$ is not an almost periodic function in  $ L^2(X,\mu_z)$ $\}$ and $H=\{z\in G: f\not \in L^2(X,\mu_z)\}$.
We need to prove that $G_1\cup H$ is measurable and has zero $\mu$-measure.
The ergodic decomposition of $\mu$ can be expressed as
 $\mu=\int_G \mu_z \dd\mu(z)$ (see e.g.\ \cite[Theorem 6.4]{M87}), thus $f\in L^2(X,\mu_z)$ for $\mu$-a.e.\ $z\in G$, $\mu(H)=0$.
 We will show $\mu(G_1)=0$ in the following.
For $k\in\N$ and $z\in G$, put
\[ F_k(z)=
\mu_z\times\mu_z\biggl(\biggl\{(x,y)\in X\times X\colon
\liminf_{n\to\infty}\frac{1}{n}\sum_{i=0}^{n-1} |f(T^ix)-f(T^iy)|>\frac{1}{k}
\biggr\}\biggr).\]
As $\int_G \mu_z \times \mu_z d\mu(z)$ is an invariant measure on $(X\times X,T\times T)$,
for each $k\in\N$, $F_k$ is a measurable function on $G$.
By Proposition~\ref{0515},
we know that
$f$ is not an almost periodic function in $L^2(X,\mu_z)$ if and only if there exists
$k\in\N$ such that $F_k(z)=1$.
Then $G_1= \bigcup_{k\in\N}\{z\in G\setminus H\colon F_k(z)=1\}$
and it is measurable.
Now it is sufficient to prove $\mu(G_1)=0$.
If not, then $\mu(G_1)>0$ and there exists $k\in\N$
such that $\mu(\{z\in G\setminus H\colon F_k(z)=1\})>0$.
Let $G_2=\{z\in G\setminus H\colon F_k(z)=1\}$ and put $r=\mu(G_2)$. Then for every $z\in G_2$ and
for $\mu_z\times\mu_z$-a.e.\ $(x,y)\in X\times X$,
\begin{equation}
\liminf_{n\to\infty}\frac{1}{n}\sum_{i=0}^{n-1}  |f(T^ix)-f(T^iy)|>\frac{1}{k}.
\label{jian-sen-33}
\end{equation}

Since $f$ is $\mu$-mean equicontinuous,
 there exists compact set $M\subset X$ with $\mu(M)>1-\frac{r^2}{4}$
such that $M$ is $f$-mean equicontinuous set.
By regularity of $\mu$, we can assume $M\subset G\setminus H$.
Let $G_3=\{z\in G\setminus H\colon \mu_z(M)>1-\frac{r}{2}\}$.
Then $G_3$ is measurable, as $\mu=\int_G \mu_z \dd\mu(z)$ is the ergodic decomposition of $\mu$.
We have
\begin{align*}
1-\frac{r^2}{4}<\mu(M)=\int_{G\setminus H} \mu_z(M)\dd\mu(z)
&\leq \int_{G_3} \mu_z(M)\dd\mu(z) +\int_{G\setminus (G_3\cup H)} \mu_z(M)\dd\mu(z) \\
&\leq \mu(G_3)+ (1-\mu(G_3))(1-\frac{r}{2}),
\end{align*}
which implies that $\mu(G_3)>1-\frac{r}{2}$.
Then $\mu(G_2\cap G_3)>r+(1-\frac{r}{2})-1=\frac{r}{2}>0$.
Pick $z\in G_2\cap G_3$. As  $M$ is an $f$-mean equicontinuous set,
there exists a $\delta>0$ such that
for any $x,y\in M$ with $d(x,y)<\delta$,
\[
\limsup_{n\to\infty}\frac{1}{n}\sum_{i=0}^{n-1} |f(T^ix)-f(T^iy)|<\frac{1}{k}.
\]
As $M$ is compact, there exists a finite open cover
$\{U_1,U_2,\dotsc, U_m\}$ of $M$ with diameter less than $\delta$.
Since $z\in G_3$, $\mu_z(M)>1-\frac{r}{2}$.
Then there exists $i\in \{1,\dotsc,m\}$ such that $\mu_z(U_i)>0$
and also $\mu_z\times\mu_z(U_i\times U_i)>0$.
Note that the diameter of $U_i$ is less than $\delta$, so
for any $x,y\in U_i$,
\[
\limsup_{n\to\infty}\frac{1}{n}\sum_{i=0}^{n-1} |f(T^ix)-f(T^iy)|<\frac{1}{k},
\]
which contradicts to (\ref{jian-sen-33}). This ends the proof of Claim 1.
\end{proof}

Let
$$G_0=\{z\in G\colon  f\ \text{ is an almost periodic function in }L^2(X,\mu_z)\}.$$
By Claim 1, we have $\mu(G_0)=1$.

\medskip
\noindent\textbf{Claim 2}:
For any $\tau>0$,
there exists $M^*\in\mathcal{B}_X$ with $\mu(M^*)>1-\tau$
such that $\{T^{n}(f\cdot \mathbf{1}_{M^*})\colon n\in\mathbb{Z}\}$  is precompact in $L^2(X,\mu)$.
\begin{proof}[Proof of the Claim 2]

Fix a constant $\tau>0$. Since $f$ is $\mu$-mean equicontinuous,
there exists compact set $M\subset X$ with $\mu(M)>1-\tau$
such that $M$ is an $f$-mean equicontinuous set. By Lusin's theorem, we can assume that $f$ is continuous on $M$.
Let $M^*=\bigcup_{n\in\mathbb{Z}} T^{-n}M$.
To show that $\{T^{n}(f\cdot \mathbf{1}_{M^*})\colon n\in\mathbb{Z}\}$  is precompact in $L^2(X,\mu)$,
we only need to prove for any sequence $\{t_n\}$ in $\mathbb{Z}$
there exists a subsequence $\{s_n\}$ of $\{t_n\}$  such that
$\{T^{s_n}(f\cdot \mathbf{1}_{M^*})\}$  is a Cauchy sequence in $L^2(X,\mu)$.

\medskip
By regularity of $\mu$, we can assume that $M\subset G_0$.
Choose a countable dense subset $\{z_m\}$ in $M$.
As $f$ is an almost periodic function in  $L^2(X,\mu_{z_1})$, there exists a subsequence $\{t_{n,1}\}$ of $\{t_n\}$ such that
$\{T^{t_{n,1}}f\colon n\in\N \}$ is a Cauchy sequence in $L^2(X,\mu_{z_1})$. Inductively assume that
for each $i\le m-1$ we have defined $\{t_{n,i}\}$ (which is a subsequence of $\{t_{n,{i-1}}\}$) such that
$\{T^{t_{n,i}}f\colon n\in\N \}$ is a Cauchy sequence in $L^2(X,\mu_{z_i})$.
As $f$ is an almost periodic function in  $L^2(X,\mu_{z_m})$,
there exists a subsequence $\{t_{n,m}\}$ of $\{t_{n,m-1}\}$ such that
$\{T^{t_{n,m}}f\colon n\in \N \}$ is a
Cauchy sequence in $L^2(X,\mu_{z_m})$.
Let $s_n=t_{n,n}$ for $n\geq 1$.
By the usual diagonal procedure,
$\{T^{s_n}f\colon n\in \N \}$ is a
Cauchy sequence in $L^2(X,\mu_{z_m})$ for all $m\geq 1$.

\medskip

Fix $\ep>0$. As $M$ is an $f$-mean equicontinuous set and $f$ is continuous on $M$,
there exists $\delta>0$ such that for any $x,y\in M$ with $d(x,y)<\delta$,
\[\limsup_{n\to\infty}\frac{1}{n}\sum_{i=0}^{n-1} |f(T^ix)-f(T^iy)|^2<\ep.\]

Since $M$ is compact and $\{z_m\}$ is dense in $M$, there exists finite set $\{z_{a_1},\cdots,z_{a_l}\}$
such that $\underset{h=1}{\overset{l}{\cup}}B(z_{a_h},\delta)\supset M$. Without loss of generality, assume $z_{a_h}=z_h$ for all $1\leq h\leq l$.
As $\{T^{s_n}f\colon n\in \N \}$ is a
Cauchy sequence in $L^2(X,\mu_{z_{h}})$ for $1\leq h\leq l$, there exists $N_{\ep}\in \N$
such that $\Vert T^{s_j}f-T^{s_k}f\Vert_{L^2(X,\mu_{z_{h}})}^2\leq \ep$ for any $j,k\geq N_{\ep}$ and $1\leq h\leq l$.

For any $z\in M$, there exists $1\leq h\leq l$ such that $d(z,z_h)<\delta$.
For any $j, k\geq N_{\ep}$,
\begin{align*}
\Vert T^{s_j}f-T^{s_k}f\Vert_{L^2(X,\mu_z)}^2&
=\int_X |T^{s_j}f-T^{s_k}f|^2 \dd\mu_z
=\lim_{n\to\infty}\frac{1}{n}\sum_{i=0}^{n-1}|f(T^{s_j+i}z)-f(T^{s_k+i}z)|^2\\
&\leq  \limsup_{n\to\infty}\frac{1}{n}\sum_{i=0}^{n-1}|f(T^{s_j+i}z)-f(T^{s_j+i}z_h)|^2\\
&+\limsup_{n\to\infty}\frac{1}{n}\sum_{i=0}^{n-1}|f(T^{s_k+i}z)-f(T^{s_k+i}z_h)|^2\\
&+\lim_{n\to\infty}\frac{1}{n}\sum_{i=0}^{n-1}|f(T^{s_j+i}z_h)-f(T^{s_k+i}z_h)|^2\\
&\leq 2\ep + \Vert T^{s_j}f-T^{s_k}f\Vert_{L^2(X,\mu_{z_h})}^2 \leq 3\ep.
\end{align*}
Thus $\{T^{s_n}f\colon n\in \N \}$ is a
Cauchy sequence in $L^2(X,\mu_{z})$.
Then for each $z\in M$,
\[ \lim_{N\to\infty}\sup_{j,k\geq N} \int_X |T^{s_j}f-T^{s_k}f|^2 \dd\mu_z =0. \]
For each $y\in M^*$, there exists $n\in\mathbb{Z}$ and $z\in M$
such that $T^nz=y$.
Then $\mu_z=\mu_y$.
For $z\in  M^*$, put
\[f_N(z)=\sup_{j,k\geq N} \int_X |T^{s_j}f-T^{s_k}f|^2 \dd\mu_z.\]
Since for any $\ep$, there exists $N_{\ep}\in \N$ such that $\Vert T^{s_j}f-T^{s_k}f\Vert_{L^2(X,\mu_{z})}^2\leq 3\ep$
for any $z\in M^*$ and any $j,k\geq N_{\ep}$,
we have $f_N(z)\leq 3\ep$ for any $z\in M^*$ and any $N\geq N_{\ep}$.
 By the dominated convergence theorem,
\begin{equation}\label{dmt}
\lim_{N\to\infty}\int_{M^*}f_N(z) \dd\mu(z)=\int_{M^*}\lim_{N\to\infty} f_N(x)\dd\mu(z)=0.
\end{equation}
It is easy to see that
\begin{align*}
\sup_{j,k\geq N}  \int_{M^*} \int_X |T^{s_j}f-T^{s_k}f|^2 \dd\mu_z \dd\mu(z)&
\le  \int_{M^*} \biggl(\sup_{j,k\geq N} \int_X |T^{s_j}f-T^{s_k}f|^2 \dd\mu_z\biggr) \dd\mu(z)\\
&=\int_{M^*}f_N(z) \dd\mu(z)
\end{align*}
which deduces
 \[ \lim_{N\to\infty}
\biggl(\sup_{j,k\geq N}  \int_{M^*} \int_X |T^{s_j}f-T^{s_k}f|^2 \dd\mu_z \dd\mu(z)\biggr)=0\ \  \ \text{by}\ (\ref{dmt}).\]
As $\int_{M^*}g \dd \mu=\int_{M^*}(\int g \dd\mu_z) \dd\mu(z)$ for any $g\in L^2(X,\mu)$ we have
 \[ \lim_{N\to\infty}
\biggl(\sup_{j,k\geq N}  \int_{M^*}  |T^{s_j}f-T^{s_k}f|^2 \dd\mu\biggr)=0. \]
Note that $T(M^*)=M^*$, so
\[\int_{M^*} |T^{s_j}f-T^{s_k}f|^2 \dd\mu=
\int |T^{s_j}(f\cdot \mathbf{1}_{M^*}) -T^{s_k}(f\cdot \mathbf{1}_{M^*})|^2 \dd\mu.\]
Thus $\{T^{s_n}(f\cdot \mathbf{1}_{M^*})\}$  is a Cauchy sequence in $L^2(X,\mu)$,
which ends the proof of Claim 2.
\end{proof}

The collection of functions $g$ such that $\{T^n g\colon n\in\mathbb{Z}\}$ is
precompact in  $L^2(X,\mu)$ is closed in $L^2(X,\mu)$.
As the measure of ${M^*}$ in Claim 2 can be arbitrarily close to $1$,
$\{T^n f\colon n\in\mathbb{Z}\}$ is also precompact in  $L^2(X,\mu)$.

\end{proof}

\begin{cor}
	Let $(X,\B,\mu,T)$ be an invertible measure preserving system where $(X,\mathcal{B},\mu)$ is a Lebesgue space, $f\in L^2(X,\mu)$. Then the following statements are equivalent:
\begin{enumerate}

\item $f$ is an almost periodic function in $L^2(X,\mu)$.

\item $f$ is $\mu$-mean equicontinuous.

\item $\mu$ has bounded  complexity
with respect to   $\{\bar{f}_n\}$.

\end{enumerate}
\end{cor}
\begin{proof} Since $(X,\mathcal{B},\mu)$ is a Lebesgue space, by Theorem \ref{0625} there is a t.d.s. $(Y,S)$
with $\nu$ an $S$-invariant Borel probability measure such that $\pi: (X,\B,\mu,T)\ra (Y,\B_Y,\nu,S)$ is measure-theoretically isomorphic.

$(2)\Rightarrow(1)$ Then $f\pi^{-1}\in  L^2(Y,\nu)$ is $\nu$-mean equicontinuous.
By theorem \ref{0623} $f\pi^{-1}$ is an almost periodic function in $L^2(Y,\nu)$. So $f$ is an almost periodic function in $L^2(X,\mu)$.

$(1)\Rightarrow(2)$ just the same as $(2)\Rightarrow(1)$.

$(2)\Rightarrow(3)$ Let $g=f\pi^{-1}$, then $g\in  L^2(Y,\nu)$ is $\nu$-mean equicontinuous.
By Theorem \ref{0622} $\nu$ has bounded  complexity
with respect to   $\{\bar{g}_n\}$, so $\mu$ has bounded  complexity
with respect to   $\{\bar{f}_n\}$.

$(3)\Rightarrow(2)$ just the same as $(2)\Rightarrow(3)$.
\end{proof}

We will show the relationship between measurable mean equicontinuity and measurable mean equicontinuous functions
which generalized Garc\'ia-Ramos and Marcus's result \cite[Theorem 3.9]{GM15} from ergodic measure cases to invariant measure cases.

\begin{thm}\label{equi001}Let $(X,T)$ be a t.d.s. with a metric $d$ and $\mu\in M(X,T)$. Then the following statements are equivalent:
\begin{enumerate}

\item $(X,T)$ is $\mu$-mean equicontinuous.

\item For any $f\in C(X)$, $f$ is $\mu$-mean equicontinuous.

\item For any $f\in L^2(X,\mu)$, $f$ is $\mu$-mean equicontinuous.

\item For any $B\in \B_X$, $1_B$ is $\mu$-mean equicontinuous.
\end{enumerate}
\end{thm}
\begin{proof}$(1)\Rightarrow(2)$ For any $\tau>0$, there is a compact set $M\subset X$ with $\mu(M)>1-\tau$ and $M$ is mean equicontinuous set.

Fix $f\in C(X)$, for any $\eta$ there exists $\theta>0$
such that $d(u,v)<\theta$ implies $|f(u)-f(v)|<\eta$.
Since $M$ is a mean equicontinuous set, there is $\delta>0$ whenever $d(x,y)<\delta$ with $x,y\in M$,
then we have $$\underline{D}\{n:d(T^nx,T^ny)< \theta\}\geq 1-\theta.$$
Then for $x,y\in M$ with $d(x,y)<\delta$, we have
 $$\underline{D}\{n:|f(T^nu)-f(T^nv)|< \eta\}\geq 1-\theta.$$
 Thus $M$ is $f$-mean equicontinuous set, and $f$ is $\mu$-mean equicontinuous.

$(2)\Rightarrow(1)$ Choose a countable set $\{f_n\}$ in $C(X)$ with $\underset{x\in X}{\max} |f_n(x)| \leq 1$ and span$\{f_n\}$ is dense in $C(X)$.
Define $$\rho(x,y)=\sum_{n=1}^\infty\frac{|f_n(x)-f_n(y)|}{2^n}.$$
Then $\rho$ is a compatible metric on $(X,T)$.
For any $\tau>0$ and  any $n\in \N$, there exists compact set $M_n$ such that $\mu(M_n)>1-\frac{\tau}{2^{n}}$ and $M_n$ is an $f_n$-mean equicontinuous set.
Let $M=\overset{\infty}{\underset{ n=1}{\cap}}M_n$, then $\mu(M)>1-\tau$ and for any $n\in \N$, $M$ is an $f_n$-mean equicontinuous set.

Next we show that $M$ is a mean equicontinuous set. For any $\ep>0$
 choose $N\in \N$ so that $\overset{\infty}{\underset{ n=N+1}{\sum}} \frac{2}{2^n}\leq \frac{\ep}{2}$. Since $M$ is an $f_n$-mean equicontinuous set,
 there exists $\delta>0$ such that whenever $d(x,y)<\delta$ with $x,y\in M$,
 $$ \limsup_{m\to\infty}\frac{1}{m}\sum_{i=0}^{m-1}|f_n(T^ix)-f_n(T^iy)|<\frac{\ep}{2}$$
 for $n\leq N$. Then
 \begin{align*}
&\hskip0.5cm \frac{1}{m}\sum_{i=0}^{m-1}|\rho(T^ix,T^iy)|\\
&\leq \frac{1}{m}\sum_{i=0}^{m-1}(\sum_{n=1}^N\frac{|f_n(T^ix)-f_n(T^iy)|}{2^n}+\frac{\ep}{2})\\
&=\sum_{n=1}^N\frac{1}{2^n}(\frac{1}{m}\sum_{i=0}^{m-1}|f_n(T^ix)-f_n(T^iy)|)+\frac{\ep}{2}
\end{align*}

 so $$\limsup_{m\to\infty}\frac{1}{m}\sum_{i=0}^{m-1}\rho(T^ix,T^iy)\leq \sum_{n=1}^N\frac{1}{2^n}\limsup_{m\to\infty}(\frac{1}{m}\sum_{i=0}^{m-1}|f_n(T^ix)-f_n(T^iy)|)+\frac{\ep}{2}$$
 $$\leq \sum_{n=1}^N\frac{1}{2^n}(\frac{\ep}{2})+\frac{\ep}{2}\leq \ep$$
Thus $M$ is a mean equicontinuous set and $(X,T)$ is $\mu$-mean equicontinuous with the metric $\rho$.
By \cite[Theorem 4.2, 4.3, 4.6]{HLTXY}, we know that $(X,T)$ is $\mu$-mean equicontinuous if and only if $(X,T)$ has discrete spectrum, which does not rely on the metric.
So $(X,T)$ is $\mu$-mean equicontinuous with metric $d$.

 $(3)\Rightarrow(4)$ and $(3)\Rightarrow(2 )$  are obvious.

 $(2)\Rightarrow(3)$
By Theorem \ref{0623}, $f$ is $\mu$-mean equicontinuous if and only if $f$ is an almost periodic function in $L^2(X,\mu)$.
Since $C(X)$ is dense in $L^2(X,\mu)$ and almost periodic functions are closed in $L^2(X,\mu)$,
 for any $f\in L^2(X,\mu)$, $f$ is $\mu$-mean equicontinuous.

  $(4)\Rightarrow(3)$   By Theorem \ref{0623}, $f$ is $\mu$-mean equicontinuous if and only if $f$ is an almost periodic function in $L^2(X,\mu)$.
  Since span$\{1_B:B\in \B\}$ is dense in $L^2(X,\mu)$ and almost periodic functions are closed in $L^2(X,\mu)$,
   for any $f\in L^2(X,\mu)$, $f$ is $\mu$-mean equicontinuous.
\end{proof}
\end{defn}

\end{document}